\documentclass[11pt]{amsart}
\input{xy}
\xyoption{all}

\newcommand{\C}{\mathcal C}
\newcommand{\Id}{\mathrm{Id}}
\newcommand{\id}{\mathrm{id}}
\newcommand{\IT}{\mathbb T}

\newcommand{\Comp}{\mathbf{Comp}}
\newcommand{\Tych}{\mathbf{Tych}}
\newcommand{\Top}{\mathbf{Top}}

\newcommand{\Mor}{\mathrm{Mor}}

\newcommand{\U}{\mathcal U}

\newtheorem{theorem}{Theorem}[section]
\newtheorem{claim}[theorem]{Claim}
\newtheorem{proposition}[theorem]{Proposition}
\newtheorem{corollary}[theorem]{Corollary}
\newtheorem{lemma}[theorem]{Lemma}
\theoremstyle{definition}
\newtheorem{definition}[theorem]{Definition}
\newtheorem{problem}[theorem]{Problem}

\title{Extending binary operations to funtor-spaces}
\author{Taras Banakh and Volodymyr Gavrylkiv}
\address{Ivan Franko National University of Lviv, Ukraine}
\address{Vasyl Stefanyk Precarpathian National University, Ivano-Frankivsk, Ukraine}
\email{t.o.banakh@gmail.com}
\email{vgavrylkiv@yahoo.com} 
 
\begin{document}
\begin{abstract}Given a continuous monadic functor $T:\Comp\to\Comp$ in the category of compacta and a discrete topological semigroup $X$ we extend the semigroup operation $\varphi:X\times X\to X$ to a right-topological semigroup operation $\Phi:T\beta X\times T\beta X\to T\beta X$ whose topological center $\Lambda_\Phi$ contains the dense subsemigroup $T_f X$ consisting of elements $a\in T\beta X$ that have finite support in $X$.
\end{abstract}

\maketitle
One of powerful tools in the modern Combinatorics of Numbers is the method of ultrafilters based on the fact that each binary operation $\varphi:X\times X\to X$ defined on a discrete topological space $X$ can be extended to a right-topological operation $\Phi:\beta X\times \beta X\to\beta X$ on the Stone-\v Cech compactification $\beta X$ of $X$, see \cite{HS}, \cite{P}. The extension of $\varphi$ is constructed in two step. First, for every $x\in X$ extend the left shift $\varphi_x:X\to X$, $\varphi_x:y\mapsto \varphi(x,y)$, to a continuous map $\overline{\varphi}_x:\beta X\to\beta X$. Next, for every $b\in\beta X$, extend the right shift $\bar \varphi^b:X\to\beta X$, $\bar\varphi^b:x\mapsto\bar\varphi_x(b)$, to a continuous map $\Phi^b:\beta X\to\beta X$ and put $\Phi(a,b)=\Phi^b(a)$ for every $a\in\beta X$. The Stone-\v Cech extension $\beta X$ is the space of ultrafilters on $X$. In \cite{G2} it was observed that the binary operation $\varphi$ extends not only to $\beta X$ but also to the superextension $\lambda X$ of $X$ and to the space $GX$ of all inclusion hyperspaces on $X$. If $X$ is a semigroup, then $GX$ is a compact Hausdorff right-topological semigroup containing $\lambda X$ and $\beta X$ as closed subsemigroups. 

In this note we show that an (associative) binary operation $\varphi:X\times X\to X$ on a discrete topological space $X$ can be extended to an (associative) right-topological operation $\Phi:T\beta X\times T\beta X\to T\beta X$ for any monadic functor $T$ in the category $\Comp$ of compact Hausdorff spaces. So, for the functors $\beta,\lambda$ or $G$, we get the extensions of the operation $\varphi$ discussed above.

\section{Monadic functors and their algebras}

Let us recall \cite[VI]{McL}, \cite[\S1.2]{TZ} that a functor $T:\C\to\C$ in a category $\C$ is called {\em monadic} if there are natural transformations $\eta:\Id\to T$ and $\mu:T^2\to T$  making the following diagrams commutative:
$$
 \xymatrix{T \ar[r]^{\eta T} \ar[d]_{T\eta} \ar[rd]^{1_{T}}&
T^2 \ar[d]^{\mu} & T^3 \ar[r]^{\mu T} \ar[d]_{T\mu} & T^2
\ar[d]^{\mu}\\
T^2 \ar[r]_{\mu} & T & T^2 \ar[r]_{\mu} & T}
$$
In this case the triple $\IT=(T,\eta,\mu)$ is called a {\em monad}, the natural transformations $\eta:\Id\to T$ and $\mu:T^2\to T$ are called the {\em unit} and {\em multiplication} of the monad $\IT$, and the functor $T$ is the {\em functorial part} of the monad $\IT$.

A pair $(X,\xi)$ consisting of an object $X$ and a morphism $\xi:TX\to X$ of the category $\C$ is called a {\em $\IT$-algebra} if $\xi\circ\eta_X=\id_X$ and the square $$\xymatrix{T^2X\ar[d]_\mu\ar[r]^{T\xi}&TX\ar[d]^\xi\\TX\ar[r]_\xi&X}$$ is commutative. For every object $X$ of the category $\C$ the pair $(TX,\mu)$ is a $\IT$-algebra called the {\em free $\IT$-algebra over $X$}. 

For two $\IT$-algebras $(X,\xi_X)$ and $(Y,\xi_Y)$ a morphism $h:X\to Y$ is called a {\em morphism of $\IT$-algebras} if the following diagram is commutative:
$$\xymatrix{TX\ar[d]_{\xi_X}\ar[r]^{Th}&TY\ar[d]^{\xi_Y}\\ X\ar[r]_h&Y}$$
The naturality of the multiplication $\mu:T^2\to T$ of the monad $\IT$ implies that for any morphism $f:X\to Y$ in $\C$ the morphism $Tf:TX\to TY$ is a morphism of  free $\IT$-algebras.  

Each morphism $h:TX\to Y$ from the free $\IT$-algebra into a $\IT$-algebra $(Y,\xi)$ is uniquely determined by the composition $h\circ\eta$.

\begin{lemma}\label{free} If $h:TX\to Y$ is a morphism of a free $\IT$-algebra $TX$ into a $\IT$-algebra $(Y,\xi)$, then $h=\mu\circ T(h\circ \eta)=\mu\circ Th\circ T\eta$. 
\end{lemma}

\begin{proof} Consider the commutative diagram $$\xymatrix{
X\ar[d]_{\eta}\ar[r]^{\eta}&TX\ar@<2pt>[d]^{\eta_{T}}\ar[r]^h&Y\\
TX\ar@/_2pc/[rr]_{T(h\circ\eta)}\ar@<2pt>[r]^{T\eta}&T^2X\ar@<2pt>[l]^{\mu}
\ar@<2pt>[u]^{\mu}\ar[r]^{Th}&TY\ar[u]_{\xi}
}
$$and observe that
$$h=h\circ \mu\circ\eta_{T}=\xi\circ Th\circ\eta_T=\xi\circ Th\circ T\eta\circ\mu\circ\eta_T=\xi\circ T(h\circ\eta).$$
\end{proof}


By a {\em topological category} we shall understand a subcategory of the category $\Top$ of topological spaces and their continuous maps such that
\begin{itemize}
\item for any objects $X,Y$ of the category $\C$ each constant map $f:X\to Y$ is a morphism of $\C$;
\item for any objects $X,Y$ of the category $\C$ the product $X\times Y$ is an object of $\C$ and for any object $Z$ of $\C$ and morphisms $f_X:Z\to X$ and $f_Y:Z\to Y$ the map $(f_X,f_Y):Z\to X\times Y$ is a morphism of the category $\C$.\end{itemize}   

A discrete topological space $X$ is called {\em discrete in $\C$} if $X$ is an object of $\C$ and each function $f:X\to Y$ into an object $Y$ of the category $\C$ is a morphism of $\C$.
It is clear that any bijection $f:X\to Y$ between discrete objects of the category $\C$ is an isomorphism in $\C$.
 
From now on we shall assume $(\IT,\eta,\mu)$ is a monad in a topological category $\C$ such that for any discrete objects $X,Y$ in $\C$ the product $X\times Y$ is discrete in $\C$.

\section{Binary operations and their $\IT$-extensions}

By a {\em binary operation in the category $\C$} we understand any function $\varphi:X\times Y\to Z$ where $X,Y,Z$ are objects of the category $\C$. 
For any $a\in X$ and $b\in Y$ the functions $$\varphi_a:Y\to Z,\;\;\varphi_a:y\mapsto\varphi(a,y)$$and
$$\varphi^b:X\to Z,\;\;\varphi^b:x\mapsto\varphi(x,b),$$are called the {\em left} and {\em right} {\em shifts}, respectively.

A binary operation $\varphi:X\times Y\to Z$ is called {\em right-topological} if for every $y\in Y$ the right shift $\varphi^y:X\to Z$, $\varphi^y:x\mapsto\varphi(x,y)$, is continuous. The {\em topological center} of a right-topological binary operation $\varphi:X\times Y\to Z$ is the set $\Lambda_\varphi$ of all elements $x\in X$ such that the left shift $\varphi_x:Y\to Z$ is continuous.

\begin{definition}\label{Text} Let $\varphi:X\times Y\to Z$ be a binary operation in the category $\C$. A binary operation $\Phi:TX\times TY\to TZ$ is defined to be a {\em $\IT$-extension} of $\varphi$ if
\begin{enumerate}
\item $\Phi(\eta_X(x),\eta_Y(y))=\eta_Z(\varphi(x,y))$ for any $x\in X$ and $y\in Y$;
\item for every $b\in TY$ the right shift $\Phi^b:TX\to TZ$, $\Phi^b:x\mapsto\Phi(x,b)$, is a morphism of the free $\IT$-algebras $TY$, $TZ$;
\item for every $x\in X$ the left shift $\Phi_{\eta(x)}:TY\to TZ$, $\Phi_{\eta(x)}:y\mapsto \Phi(\eta(x),y)$, is a morphism of the free $\IT$-algebras $TX$, $TZ$.
\end{enumerate}
\end{definition}

This definition implies that for any binary operation $\varphi:X\times Y\to Z$ its $\IT$-extension $\Phi:TX\times TY\to TZ$ is a right-topological binary operation whose topological center $\Lambda_\Phi$ contains the set $\eta(X)\subset TX$.

\begin{theorem}\label{unique} Let $\varphi:X\times Y\to Z$ be a binary operation in the category $\C$.
\begin{enumerate} 
\item The binary operation $\varphi$ has at most one $\IT$-extension $\Phi:TX\times TY\to TZ$.
\item If $X,Y$ are discrete in $\C$, then $\varphi$ has a unique $\IT$-extension $\Phi:TX\times TY\to TZ$.
\end{enumerate}
\end{theorem}

\begin{proof} 1. Let $\Phi,\Psi:TX\times TY\to TZ$ be two $\IT$-extensions of the operation $\varphi$. By the condition (3) of Definition~\ref{Text}, for every $x\in X$ and $a=\eta_X(x)\in TX$, the left shifts $\Phi_a,\Psi_a:TY\to TZ$ are morphisms of the free $\IT$-algebras.  

By the condition (1) of Definition~\ref{Text}, $$\Phi_a\circ \eta_Y=\eta_Z\circ\varphi_x=\Psi_a\circ\eta_Y.$$ Then Lemma~\ref{free} implies that $$\Phi_a=\mu\circ T(\Phi_a\circ\eta_X)=\mu\circ T(\eta_Z\circ\varphi_x)=\mu\circ T(\Psi_a\circ\eta_X)=\Psi_a.$$

The equality $\Phi=\Psi$ will follow as soon as we check that $\Phi^b=\Psi^b$ for every $b\in TY$. Since $\Phi^b,\Psi^b:TX\to TZ$ are morphisms of the free $\IT$-algebras $TX$ and $TZ$, the equality $\Phi^b=\Psi^b$ follows from 
the equality $$\Phi^b\circ\eta(x)=\Phi_{\eta(x)}(b)=\Psi_{\eta(x)}(b)=\Psi^b\circ\eta(x),\;\;x\in X$$according to Lemma~\ref{free}.
\smallskip

2. Now assuming  that the spaces $X,Y$ are discrete in $\C$, we show that the binary operation $\varphi:X\times Y\to Z$ has a $\IT$-extension. For every $x\in X$ consider the left shift $\varphi_x:Y\to Z$. Since  $Y$ is discrete in $\C$, the function $\varphi_x$ is a morphism of the category $\C$. Applying the functor $T$ to this morphism, we get a morphism $T\varphi_x:TY\to TZ$. Now for every $b\in TY$ consider the function $\varphi^b:X\to TZ$, $\varphi^b:x\mapsto T\varphi_x(b)$. Since the object $X$ is discrete, the function $\varphi^b$ is a morphism of the category $\C$. Applying to this morphism the functor $T$, we get a morphism $T\varphi^b:TX\to T^2Z$. Composing this morphism with the multiplication $\mu:T^2Z\to TZ$ of the monad $\IT$, we get the function $\Phi^b=\mu\circ T\varphi^b:TZ\to TZ$. Define a binary operation $\Phi:TX\times TY\to TZ$ letting $\Phi(a,b)=\Phi^b(a)$ for $a\in TX$. 

\begin{claim}\label{cl1}$\Phi(\eta(x),b)=T\varphi_x(b)$ for every $x\in X$ and $b\in TY$.
\end{claim}

\begin{proof} The commutativity of the diagram
$$\xymatrix{
X\ar[d]_{\eta}\ar[r]^{\varphi^b}&TZ\ar[d]_{\eta}\\
TX\ar[r]_{T\varphi^b}\ar[ur]^{\Phi^b\!\!\!}&T^2Z\ar@/_1pc/[u]_{\mu}
}
$$implies the desired equality $$\Phi(\eta(x),b)=\mu\circ T\varphi^b(\eta(x))=\varphi^b(x)=T\varphi_x(b).$$
\end{proof}

Now we shall prove that $\Phi$ is a $\IT$-extension of $\varphi$.
\smallskip

i) For every $x\in X$ and $y\in Y$ we need to prove the equality
$$\Phi(\eta_X(x),\eta_Y(y))=\eta_Z\circ\varphi(x,y).$$ 
By Claim~\ref{cl1},
$$\Phi(\eta_X(x),\eta_Y(y))=T\varphi_x\circ \eta_Y(y)=\eta_Z\circ\varphi_x(y)=\eta_Z\circ \varphi(x,y).$$
The latter equality follows from the naturality of the transformation $\eta:\Id\to T$.
\smallskip

ii) The definition of $\Phi$ implies that for every $b\in TY$ the right shift $\Phi^b=\mu_Z\circ T\varphi^b$ is a morphism of free $\IT$-algebras, being the compositions of two morphisms $T\varphi^b:TX\to T^2Z$ and $\mu_Z:T^2Z\to TZ$ of free $\IT$-algebras.
\smallskip

iii) Claim~\ref{cl1} guarantees that for every $x\in X$ the left shift $\Phi_{\eta(x)}=T\varphi_x:TY\to TZ$ is a morphism of the free $\IT$-algebras.
\end{proof}

\begin{proposition} Let $\varphi:X\times Y\to Z$, $\psi:X'\times Y'\to Z'$ be  two binary operations in $\C$, $\Phi:TX\times TY\to TZ$, $\Psi:TX'\times TY'\to TZ'$ be their $\IT$-extensions, and $h_X:X\to X'$, $h_Y:Y\to Y'$, $h_Z:Z\to Z'$ be morphisms in $\C$. If $\psi(h_X\times h_Y)=h_Z\circ \phi$, then $T\Psi(Th_X\times Th_Y)=Th_Z\circ \Phi$.
\end{proposition}

\begin{proof} Observe that for any $x\in X$ and $x'=h_X(x)$, the commutativity of the diagrams
$$\xymatrix{
Y\ar[r]^{\varphi_x}\ar[d]_{h_Y}&Z\ar[d]^{h_Z}\\
Y'\ar[r]_{\psi_{x'}}&Z'
}\quad\quad\quad
\xymatrix{
TY\ar[r]^{T\varphi_x}\ar[d]_{Th_Y}&TZ\ar[d]^{Th_Z}\\
TY'\ar[r]_{T\psi_{x'}}&TZ'
}
$$
imply that $Th_Z\circ T\varphi_x(b)=T\psi_{x'}(b')$
for every $b\in TY$ and $b'=Th_Y(b)\in TY'$.

It follows from Lemma~\ref{cl1} that $\Phi_{\eta(x)}=T\varphi_x:TY\to TZ$ and $\Psi_{\eta(x')}=T\psi_{x'}:TY'\to TZ'$. Consequently, $$Th_Z\circ \Phi^b(\eta(x))=Th_Z\circ \Phi_{\eta(x)}(b)=Th_Z\circ T\varphi_x(b)=T\psi_{x'}(b')=\Psi_{\eta(x')}(b')=\Psi^{b'}(\eta(x'))$$and hence $$Th_Z\circ\Phi^b\circ\eta=\Psi^{b'}\circ \eta\circ h_X.$$  Applying the functor $T$ to this equality, we get $$T^2h_Z\circ T(\Phi^b\circ\eta)=T(\Psi^{b'}\circ\eta)\circ Th_X.$$ Since $\Phi^b:TX\to TZ$ and $\Psi^{b'}:TX'\to TZ'$ are homomorphisms of the free $\IT$-algebras, we can apply Lemma~\ref{free} and conclude that $\Phi^b=\mu\circ T(\Phi^b\circ\eta)$ and hence
$$Th_Z\circ\Phi^b=Th_Z\circ\mu_Z\circ T(\Phi^b\circ\eta)=\mu_{Z'}\circ T^2h_Z\circ T(\Phi^b\circ\eta)=\mu_{Z'}\circ T(\Psi^{b'}\circ\eta)\circ Th_X=\Psi^{b'}\circ Th_X.$$
Then for every $a\in TX$ we get
$$Th_Z\circ\Phi(a,b)=Th_Z\circ\Phi^b(a)=\Psi^{b'}\circ Th_X(a)=\Psi(Th_X(a),Th_Y(b)).$$
\end{proof}

\section{Binary operations and tensor products}

In this section we shall discuss the relation of $\IT$-extensions to tensor products. The tensor product is a function $\otimes:TX\times TY\to T(X\times Y)$ defined for any objects $X,Y\in\C$ such that $X$ is discrete in $\C$.

For every $x\in X$ consider the embedding $i_x:Y\to X\times Y$, $i_x:y\mapsto (x,y)$. The embedding $i_x$ is a morphism of the category $\C$ because the constant map $c_x:Y\to\{x\}\subset X$ and the identity map $\id:Y\to Y$ are morphisms of the category and $\C$ contains products of its objects. Applying the functor $T$ to the morphism 
$i_x$, we get a morphism $Ti_x:TY\to T(X\times Y)$ of the category $\C$. Next, for every $b\in TY$ consider the function $Ti^b:X\to T(X\times Y)$, $Ti^b:x\mapsto Ti_x(b)$. Since $X$ is  discrete in $\C$, the function $Ti^b$ is a morphism of the category $\C$. Applying the functor $T$ to this morphism, we get a morphism $TTi^b:TX\to T^2(X\times Y)$. Composing this morphism with the multiplication $\mu:T^2(X\times Y)\to T(X\times Y)$ of the monad $\IT$, we get the morphism $\otimes^b=\mu\circ TTi^b:TX\to T(X\times Y)$. Finally define the tensor product $\otimes:TX\times TY\to T(X\times Y)$ letting $a\otimes b=\otimes^b(a)$ for $a\in TX$.

The following proposition describes some basic properties of the tensor product. 
For monadic functors in the category $\mathbf{Comp}$ of compact Hausdorff spaces those properties were established in \cite[3.4.2]{TZ}.

\begin{proposition}\label{tensor} 
\begin{enumerate}
\item The diagram $\xymatrix@1{X\times Y\ar@/^1pc/[rr]^{\eta}\ar[r]_-{\eta\times\eta}&TX\times TY\ar[r]_-\otimes &T(X\times Y)}$ is commutative for any discrete object $X$ and any object $Y$ of $\C$;
\item the tensor product is natural in the sense that for any morphisms $h_X:X\to X'$, $h_Y:Y\to Y'$ of $\C$ with discrete $X,Y$, the following diagram is commutative:
$$\xymatrix{
TX\times TY\ar[d]_{Th_X\times Th_Y}\ar[r]^{\otimes}&T(X\times Y)\ar[d]^{T(h_X\times h_Y)}\\
TX'\times TY'\ar[r]^{\otimes}& T(X'\times Y')}
$$
\item the tensor product is associative in the sense that for any discrete objects $X,Y,Z$ of $\C$  the diagram 
$$\xymatrix{TX\times TY\times TZ\ar[d]_{\id\times\otimes}\ar[r]^{\otimes\times \id}&T(X\times Y)\times TZ\ar[d]^{\otimes}\\
TX\times T(Y\times Z)\ar[r]_{\otimes}&T(X\times Y\times Z)}
$$is commutative, which means that $(a\otimes b)\otimes c=a\otimes(b\otimes c)$ for any $a\in TX$, $b\in TY$, $c\in TZ$.
\end{enumerate}
\end{proposition}

\begin{proof} 1. Fix any $y\in Y$ and consider the element $b=\eta_Y(y)\in TY$.  The definition of the right shift $\otimes^b$ implies that the following diagram is commutative:
$$\xymatrix{
X\ar[d]_{\eta}\ar[r]^-{Ti^b}&T(X\times Y)\\
TX\ar[r]_-{TTi^b}\ar[ru]^{\otimes^b}&T^2(X\times Y)\ar[u]_{\mu}}
$$
Consequently, for every $x\in X$ we get $$\eta(x)\otimes \eta(y)=\otimes^b\circ\eta(x)=Ti^b\circ\eta(x)=Ti_x(\eta(y))=\eta(i_x(y))=\eta(x,y).$$ The latter equality follows from the diagram
$$\xymatrix{Y\ar[d]_\eta\ar[r]^-{i_x}&X\times Y\ar[d]^{\eta}\\
TY\ar[r]_-{Ti_x}&T(X\times Y)}
$$whose commutativity follows from the naturality of the transformation $\eta:\Id\to T$. 
\smallskip

2. Let $h_X:X\to X'$ and $h_Y:Y\to Y'$ be any functions between discrete objects of the category $\C$. Let $Z=X\times Y$, $Z'=X'\times Y'$ and $h_Z=h_X\times h_Y:Z\to Z'$. Given any point $b\in TY$, consider the element $b'=Th_Y(b)\in TY'$. The statement (2) will follow as soon as we check that $Th_Z\circ\otimes^b=\otimes^{b'}\circ Th_X$.
By Lemma~\ref{free}, this equality will follow as soon as we check that $Th_Z\circ\otimes^b\circ\eta_X=\otimes^{b'}\circ Th_X\circ \eta_X=\otimes^{b'}\circ \eta_{X'}\circ h_X$. The last equality follows from the naturality of the transformation $\eta:\Id\to T$. As we know from the proof of the preceding item, 
$\otimes^{b'}\circ \eta_{X'}(x')=Ti_{x'}(b')$ for any $x'\in X'$. For every $x\in X$ and $x'=h_X(x)$ we can apply the functor $T$ to the commutative diagram
$$\xymatrix{
Y\ar[d]_{h_Y}\ar[r]^{i_x}&Z\ar[d]^{h_Z}\\
Y'\ar[r]_{i_{x'}}&Z'}
$$and obtain the equality $Th_Z\circ Ti_x=Ti_{x'}\circ Th_Y$ which implies the desired equality:
$$\otimes^{b'}\circ\eta_{X'}\circ h_X(x)=\otimes^{b'}\circ\eta_{X'}(x')=Ti_{x'}(b')=Th_Z\circ Ti_x(b)=Th_Z\circ\otimes^b\circ \eta(x).$$

3. The proof of the associativity of the tensor product can be obtained by literal rewriting the proof of Proposition 3.4.2(4) of \cite{TZ}.
\end{proof}

\begin{theorem}\label{ext-tensor} Let  $\varphi:X\times Y\to Z$ be a binary operation in the category $\C$ and $\Phi:TX\times TY\to TZ$ be its $\IT$-extension. If  $X$ is a discrete object in $\C$, then $\Phi(a,b)=T\varphi(a\otimes b)$ for any elements $a\in TX$ and $b\in TY$.
\end{theorem}

\begin{proof} Our assumptions on the category $\C$ guarantee that the product  $X\times Y$ is a discrete object of $\C$ and hence $\varphi:X\times Y\to Z$ is a morphism of the category $\C$. So, it is legal to consider the morphism $T\varphi:T(X\times Y)\to TZ$. We claim that the binary operation $$\Psi:TX\times TY\to TZ,\;\;\Psi:(a,b)=T\varphi(a\otimes b),$$
is a $\IT$-extension of $\varphi$.
\smallskip

1. The first item of Definition~\ref{Text} follows Proposition~\ref{tensor}(1) and the naturality of the transformation $\eta:\Id\to T$:
$$\Psi(\eta_X(x),\eta_Y(y))=T\varphi(\eta_X(x)\otimes \eta_Y(y))=T\varphi\circ\eta_{X\times Y}(x,y)=\eta_Z\circ\varphi(x,y).$$
\smallskip

2. For every $b\in TY$ the morphism $$\Psi^b=T\varphi\circ\otimes^b=T\varphi\circ \mu\circ TTi^b$$ is a morphism of the free $\IT$-algebras $TX$ and $TZ$.
\smallskip

3. For every $x\in X$ we see that $$\Psi_{\eta(x)}(b)=T\varphi(\otimes^b(\eta(x)))=T\varphi\circ\mu\circ TTi^b\circ\eta(x)=T\varphi\circ\mu\circ \eta\circ Ti^b(x)=T\varphi\circ Ti^b(x)$$
is a morphism of the free $\IT$-algebras $TY$ and $TZ$.
\smallskip

Thus $\Psi$ is a $\IT$-extension of the binary operation $\varphi$. By the Uniqueness Theorem~\ref{unique}(1), $\Psi$ coincides with $\Phi$ and hence $\Phi(a,b)=\Psi(a,b)=T\varphi(a\otimes b)$.
\end{proof}

\section{The topological center of $\IT$-extended operation}

Definition~\ref{Text} guarantees that for a binary operation $\varphi:X\times Y\to Z$  in $\C$ any $\IT$-extension $\Phi:TX\times TY\to TZ$ of $\varphi$ is a right-topological operation whose topological center $\Lambda_\varphi$ contains the subset $\eta_X(X)$. In this section we shall find conditions on the functor $T$ and the space $X$ guaranteeing that the topological center $\Lambda_\Phi$ is dense in $TX$.

We shall say that the functor $T$ is {\em continuous} if for each compact Hausdorff space $K$ that belongs to the category $\C$ and any object $Z$ of $\C$ the map $T:\Mor(K,Z)\to\Mor(TK,TZ)$, $T:f\mapsto Tf$, is continuous with respect to the compact-open topology on the spaces of morphisms (which are continuous maps). 

\begin{theorem}\label{cont} Let $\varphi:X\times Y\to Z$ be a binary operation in $\C$ and $\Phi:X\times Y\to Z$ be its $\IT$-extension. If the object $X$ is finite and discrete in $\C$, $TX$ is locally compact and Hausdorff, and the functor $T$ is continuous, then the operation $\Phi$ is continuous. 
\end{theorem}

\begin{proof} Since the space $X$ is discrete, the condition (2) of Definition~\ref{Text} implies that the map $\Phi_\eta:X\times TY\to TZ$, $\Phi_\eta:(x,b)\mapsto \Phi(\eta(x),b)$ is continuous. Since $X$ is finite, the induced map $$\Phi_\eta^{(\cdot)}:TY\to \Mor(X,TZ),\;\;\Phi_\eta^{(\cdot)}:b\mapsto\Phi_\eta^b\mbox{ \ where \ } \Phi_\eta^b:x\mapsto \Phi(\eta(x),b),$$ is continuous. By the continuity of the functor $T$, the map $T:\Mor(X,TZ)\to \Mor(TX,T^2Z)$, $T:f\mapsto Tf$, is continuous and so is the  composition $T\circ\Phi_\eta^{(\cdot)}:TY\to \Mor(TX,T^2Z)$. Since $TX$ is locally compact and Hausdorff, we can apply \cite[3.4.8]{En} and conclude that the map $$T\Phi^{(\cdot)}_\eta:TX\times TY\to T^2Z,\;\;T\Phi_{\eta}^{(\cdot)}:(a,b)\mapsto T\Phi^b_\eta(a)$$ is continuous and so is the composition $\Psi=\mu\circ T\Phi_\eta^{(\cdot)}:TX\times TY\to TZ$. Using the Uniqueness Theorem~\ref{unique}(1), we can prove that $\Psi=\Phi$ and hence the binary operation $\Phi$ is continuous.
\end{proof}

Let $X$ be an object of the category $\C$. We say that an element $a\in FX$ has {\em discrete (finite) support} if there is a morphism $f:D\to X$ from a  discrete (and finite) object $D$ of the category $\C$ such that $a\in Ff(FD)$. By $T_d X$  (resp. $T_f X$) we denote the set of all elements $a\in TX$ that have discrete (finite) support. It is clear that $T_f X\subset T_d X\subset TX$.

\begin{theorem}\label{top-cent} Let $\varphi:X\times Y\to Z$ be a binary operation and $\Phi:TX\times TY\to TZ$ be a $\IT$-extension of $\varphi$. If the functor $T$ is continuous, and for every finite discrete object $D$ of $\C$ the space $TD$ is locally compact and Hausdorff, then the topological center $\Lambda_\Phi$ of the binary operation $\Phi$ contains the subspace $T_f X$ of $TX$. If $T_f X$ is dense in $TX$, then the topological center $\Lambda_\Phi$ of $\Phi$ is dense in $TX$.
\end{theorem}

\begin{proof} We need to prove that for every $a\in T_f X$ the left shift $\Phi_a:TY\to TZ$, $\Phi_a:b\mapsto \Phi(a,b)$, is continuous. Since $a\in T_f X$, there is a finite discrete object $D$ of the category $\C$ and a morphism $f:D\to X$ such that $a\in Ff(FD)$. Fix an element $d\in FD$ such that $a=Ff(d)$.

Consider the binary operations $$\psi:D\times Y\to Z,\;\psi:(x,y)\mapsto \phi(f(x),y),$$ and $$\Psi:TD\times TY\to TZ,\;\Psi:(a,b)\mapsto \Phi(Ff(a),b).$$
It can be shown that $\Psi$ is a $\IT$-extension of $\psi$. 

By Theorem~\ref{cont}, the binary operation $\Psi$ is continuous. Consequently, the left shift $\Psi_d:TY\to TZ$, $\Psi_d:b\mapsto\Psi(d,b)$, is continuous. Since $\Psi_d=\Phi_a$, the left shift $\Phi_a$ is continuous too and hence $a\in \Lambda_\Phi$.
\end{proof}

\section{The associativity of $\IT$-extensions}

In this section we investigate the associativity of the $\IT$-extensions. We recall that a binary operation $\varphi:X\times X\to X$ is {\em associative} if $\varphi(\varphi(x,y),z)=\varphi(x,\varphi(y,z))$ for any $x,y,z\in X$. In this case we say that $X$ is a {\em semigroup}.

 A subset $A$ of a set $X$ endowed with a binary operation $\varphi:X\times X\to X$ is called a {\em subsemigroup} of $X$ if $\varphi(A\times A)\subset A$ and $\varphi(\varphi(x,y),z)=\varphi(x,\varphi(y,z))$ for all $x,y,z\in A$.

 \begin{lemma}\label{asl} Let $\varphi:X\times X\to X$ be an associative operation in $\C$ and $\Phi:TX\times TX\to TX$ be its $\IT$-extension.
\begin{enumerate}
\item for any morphisms $f_A:A\to X$, $f_B:B\to X$ from discrete objects $A,B$ in $\C$, the map $\varphi_{AB}=\varphi(f_A\times f_B):A\times B\to X$ is a morphism of $\C$ such that $\Phi(Tf_A(a),Tf_B(b))=T\varphi_{AB}(a\otimes b)$ for all $a\in TA$ and $b\in TB$;
\item $\Phi(T_d X\times T_d X)\subset T_d X$ and $\Phi(T_f X\times T_f X)\subset T_f X$;
\item $\Phi((a,b),c)=\Phi(a,\Phi(b,c))$ for any $a,b,c\in T_d X$.
\end{enumerate}
\end{lemma}

\begin{proof} 
1. Let $f_A:A\to X$, $f_B:B\to X$ be morphisms from discrete objects $A,B$ of $\C$ and $\varphi_{AB}=\varphi(f_A\times f_B):A\times B\to X$. By our assumption on the category $\C$, the product $A\times B$ is a discrete object in $\C$ and hence $\phi_{AB}$ is a morphism in $\C$. Consider the binary operation $\Phi_{AB}:TA\times TB\to TX$ defined by $\Phi_{AB}(a,b)=\Phi(Tf_A(a),Tf_B(b))$. The following diagram implies that $\Phi_{AB}$ is a $\IT$-extension of $\varphi_{AB}$: 
$$\xymatrix{
TX\times TX\ar[rrr]^{\Phi}& & &TX\\
& X\times X\ar[ul]_{\eta\times\eta}\ar[r]^-{\varphi}&X\ar[ur]^{\eta}&\\
&A\times B\ar[dl]^{\eta\times\eta}\ar[u]^{f_A\times f_B}\ar[r]^-{\varphi_{AB}}&X\ar@{<->}[u]_{\id}\ar[rd]_{\eta}&\\
TA\times TB\ar[uuu]^{Tf_A\times Tf_B}\ar[rrr]_{\Phi_{AB}}&&&TX\ar@{<->}[uuu]_{\id}
}$$
By Theorem~\ref{ext-tensor}, $$\Phi(Tf_A(a),Tf_B(b))=\Phi_{AB}(a,b)=T\varphi_{AB}(a\otimes b)$$for all $a\in TA$ and $b\in TB$.
\smallskip

2. Given elements $a,b\in T_d X$, we need to show that the element  $\Phi(a,b)\in TX$ has discrete support. Find discrete objects $A,B$ in $\C$ and morphisms $f_A:A\to X$, $f_B:B\to X$ such that  $a\in Ff_A(FA)$ and $b\in f_B(FB)$.
Fix elements $\tilde a\in FA$, $\tilde b\in FB$ such that $a=Ff_A(\tilde a)$ and $b=Ff_B(\tilde b)$. Our assumption on the category $\C$ guarantees that $A\times B$ is a discrete object in $\C$.  

Consider the binary operations $\psi:A\times B\to X$ and $\Psi:FA\times FB\to FZ$ defined by the formulae $\psi=\phi\circ(f_A\times f_B)$ and $\Psi=\Phi\circ (Tf_A\times Tf_B)$. 
Let $\tilde c=\tilde a\otimes\tilde b\in T(A\times B)$. 
By the first statement, $\Phi(a,b)=T\psi(\tilde a\otimes \tilde b)=T\psi(\tilde c)\in T\psi(A\times B)$ witnessing that the element $\Phi(a,b)$ has discrete support and hence belongs to $T_d X$.

By analogy, we can prove that $\Phi(T_f X\times T_f X)\subset T_f X$. 
\smallskip

3. Given any points $a,b,c\in T_d X$, we need to check the equality
$$\Phi(\Phi(a,b),c)=\Phi(a,\Phi(b,c)).$$

Find discrete objects $A,B,C$ in $\C$ and morphisms $f_A:A\to X$, $f_B:B\to X$, $f_C:C\to X$ such that $a\in Tf_A(TA)$, $b\in Tf_B(TB)$, and $c\in Tf_C(TC)$.  Fix elements $\tilde a\in TA$, $\tilde b\in TB$, and $\tilde c\in TC$ such that $a=Tf_A(\tilde a)$, $b=Tf_B(\tilde b)$, and $c=Tf_C(\tilde c)$.

Consider the morphisms $\varphi_{AB}=\varphi(f_A\times f_B):A\times B\to X$, $\varphi_{BC}=\varphi(f_B\times f_C):B\times C\to X$ and  $\varphi_{ABC}=\varphi(\varphi_{AB}\times f_C)=\varphi(f_A\times\varphi_{BC}):A\times B\times C\to X$.
Consider the following diagram:
$$\xymatrix{
TX\times TX\times TX\ar[ddd]_{\id\times\Phi}\ar[rrr]^{\Phi\times\id}& & & TX\times TX\ar[ddd]^{\Phi}\\
& TA\times TB\times TC\ar[lu]_{Tf_A\times Tf_B\times Tf_C}\ar[d]_{\id\times\otimes}\ar[r]^{\otimes\times \id}&T(A\times B)\times TC\ar[d]^{\otimes}\ar[ru]^{T\varphi_{AB}\times Tf_C}&\\
&TA\times T(B\times C)\ar[ld]^{Tf_A\times T\varphi_{BC}}\ar[r]_\otimes &T(A\times B\times C)\ar[rd]_{T\varphi_{ABC}}&\\
TX\times TX\ar[rrr]_{\Phi}& & & TX
}$$
In this diagram the central square is commutative because of the associativity of the tensor product $\otimes$. By the item (1) all four margin squares also are commutative. Now we see that 
$$\begin{aligned}
&\Phi(\Phi(a,b),c))=\Phi(\Phi(Tf_A(\tilde a),Tf_B(\tilde b)),Tf_C(\tilde c))=\\ &\Phi(T\varphi_{AB}(\tilde a\otimes \tilde b),Tf_C(\tilde c))=T\varphi_{ABC}((\tilde a\otimes \tilde b)\otimes \tilde c)=T\varphi_{ABC}(\tilde a\otimes(\tilde b\otimes \tilde c))=\\
&\Phi(Tf_A(\tilde a),T\varphi_{BC}(\tilde a\otimes \tilde b))=\Phi(Tf_A(\tilde a),\Phi(Tf_B(\tilde b),Tf_C(\tilde c)))=\Phi(a,\Phi(b,c)).
\end{aligned}$$
\end{proof}

Combining Lemma~\ref{asl} with Theorem~\ref{top-cent}, we get the main result of this paper:

\begin{theorem}\label{ast} Assume that the monadic functor $T$ is continuous and for each finite discrete space $F$ in $\C$ the space $TF$ is Hausdorff and locally compact. Let $\varphi:X\times X\to X$ be an associative binary operation in $\C$ and $\Phi:X\times X\to X$ be its $\IT$-extension. If the set $T_f X$ of elements with finite support is dense in $TX$, then the operation $\Phi$ is associative.
\end{theorem}

\begin{proof} By Theorem~\ref{top-cent}, the set $T_f X$ lies in the topological center $\Lambda_\Phi$ of the operation $\Phi$ and by Lemma~\ref{asl}, $T_f X$ is a subsemigroup of $(TX,\Phi)$. Now the associativity of $\Phi$ follows from the following general fact.
\end{proof} 

\begin{proposition}\label{asp} A right topological operation $\cdot:X\times X\to X$ on a Hausdorff space $X$ is associative if its topological center  contains a dense subsemigroup $S$ of $X$.
\end{proposition}

\begin{proof} Assume conversely that $(xy)z\ne x(yz)$ for some points $x,y,z\in X$. Since $X$ is Hausdorff, the points $(xy)z$ and $x(yz)$ have disjoint open neighborhoods $O((xy)z)$ and $O(x(yz))$ in $X$. Since the right shifts in $X$ are continuous, there are open neighborhoods $O(xy)$ and $O(x)$ of the points $xy$ and $x$ such that $O(xy)\cdot z\subset O((xy)z)$ and $O(x)\cdot(yz)\subset O(x(yz))$. We can assume that $O(x)$ is so small that $O(x)\cdot y\subset O(xy)$. Take any point $a\in O(x)\cap S$. It follows that $a(yz)\in O(x(yz))$ and $ay\in O(xy)$. Since the left shift $l_a:\beta S\to\beta S$, $l_a:y\mapsto ay$, is continuous, the points $yz$ and $y$ have open neighborhoods $O(yz)$ and $O(y)$ such that $a\cdot O(yz)\subset O(x(yz))$ and $a\cdot O(y)\subset O(xy)$. We can assume that the neighborhood $O(y)$ is so small that $O(y)\cdot z\subset O(yz)$. 
Choose a point $b\in O(y)\cap S$ and observe that $bz\in O(y)\cdot z\subset O(yz)$, $ab\in a\cdot O(y)\subset O(xy)$, and thus $(ab)z\in O(xy)\cdot z\subset O((xy)z)$. The continuity of the left shifts $l_b$ and $l_{ab}$ allows us to find an open neighborhood $O(z)\subset\beta S$ of $z$ such that $b\cdot O(z)\subset O(yz)$ and $ab\cdot O(z)\subset O((xy)z)$. Finally take any point $c\in S\cap O(z)$. Then $(ab)c\in ab\cdot O(z)\subset O((xy)z)$ and $a(bc)\subset a\cdot O(yz)\subset O(x(yz))$ belong to disjoint sets, which is not possible as  $(ab)c=a(bc)$.
\end{proof}

\section{$\IT$-extension for some concrete monadic functors}

In this section we consider some examples of monadic functors in topological categories. Let $\Tych$ denote the category of Tychonov spaces and their continuous maps and $\Comp$ be the full subcategory of the category $\Tych$, consisting of compact Hausdorff spaces. 

Discrete objects in the category $\Tych$ are discrete topological spaces while discrete objects in the category $\Comp$ are finite discrete spaces. 

Consider the functor $\beta:\Tych\to \Comp$ assigning to each Tychonov space $X$ its Stone-\v Cech compactification and to a continuous map $f:X\to Y$ between Tychonov spaces its continuous extension $\beta f:\beta X\to \beta Y$. The functor $\beta$ can be completed to a monad $\IT_\beta=(\beta,\eta,\mu)$ where $\eta:X\to\beta X$ is the canonical embedding and $\mu:\beta(\beta X)\to\beta X$  is the identity map. A pair $(X,\xi)$ is a $\IT_\beta$-algebra if and only if $X$ is a compact space and $\xi:\beta X\to X$ is the identity map. 

Combining Theorems~\ref{unique}, \ref{ast} we get the following well-known 

\begin{corollary} Each binary right-topological operation $\varphi:X\times Y\to Z$ in $\Tych$ with discrete $X$ can be extended to a right-topological operation $\Phi:\beta X\times\beta Y\to\beta Z$ containing $X$ in its topological center $\Lambda_\Phi$. 
If $X=Y=Z$ and the operation $\varphi$ is associative, then so is the operation $\Phi$.
\end{corollary} 

Now let $\IT=(T,\eta,\mu)$ be a monad in the category $\Comp$. Taking the composition of the functors $\beta:\Tych\to\Comp$ and $T:\Comp\to\Comp$, we obtain a monadic functor $T\beta:\Tych\to\Comp$. 

\begin{theorem} Each binary right-topological operation $\varphi:X\times Y\to Z$ in the category $\Tych$ with discrete $X$ can be extended to a right-topological operation $\Phi:T\beta X\times T\beta Y\to T\beta Z$ that contain the set $\eta(X)\subset T\beta X$ in its topological center $\Lambda_\Phi$. If the functor $T$ is continuous, then the set $T_f X$ of elements $a\in T\beta X$ with finite support is dense in $T\beta X$ and lies in the topological center $\Lambda_\Phi$ of the operation $\Phi$. Moreover, if $X=Y=Z$ and the operation $\varphi$ is associative, the so is the operation $\Phi$. 
\end{theorem} 

\begin{proof} By Theorem~\ref{unique}, the binary operation $\varphi$ has a unique $\IT$-extension $\Phi:TX\times TY\to TZ$. By Definition~\ref{Text}, the set $\eta(X)\subset T\beta X$ lies in the topological center $\Lambda_\varphi$ of $\varphi$.

Now assume that the functor $T$ is continuous. First we show that the set $T_f X$ is dense in $T\beta X$. Fix any point $a\in F\beta X$ and an open neighborhood $U\subset T\beta X$ of $a$. Then $[a,U]=\{f\in \Mor(F\beta X,F\beta X):f(a)\in U\}$ is an open neighborhood of the identity map $\id:F\beta X\to F\beta X$ in the function space $\Mor(F\beta X,F\beta X)$ endowed with the compact-open topology. The continuity of the functor $T$ yields a neighborhood $\U(\id_{\beta X})$ of the identity map $\id_{\beta X}\in\Mor(\beta X,\beta X)$ such that $Tf\in[a,U]$ for any $f\in\U(\id_{\beta X})$. It follows from the definition of the compact-open topology, that there is an open cover $\U$ of $\beta X$ such that a map $f:\beta X\to\beta X$ belongs to $\U(\id_{\beta X})$ if $f$ is $\U$-near to $\id_{\beta X}$ in the sense that for every $x\in\beta X$ there is a set $U\in\U$ with $\{x,f(x)\}\subset U$. Since $\beta X$ is compact, we can assume that the cover $\U$ is finite. Since $X$ is discrete, the space $\beta X$ has covering dimension zero \cite[7.1.17]{En}. So, we can assume that the finite cover $\U$ is disjoint. For every $U\in\U$ choose an element $x_U\in U\cap X$. Those elements compose a finite discrete subspace $A=\{x_U:U\in \U\}$ of $X$.
let $i:A\to X$ be the identity embedding and $f:X\to A$ be the map defined by $f^{-1}(x_U)=U$ for $U\in\U$. It follows that $i\circ f\in\U(\id_{\beta X})$ and thus $T(i\circ f)\in [a,U]$ and $Ti\circ Tf(a)\in U$. Now we see that $b=Tf(a)\in TA$ and $c=Ti(b)\in T_f X\cap U$, so $T_f X$ is dense in $\beta X$.

By Theorem~\ref{top-cent}, the set $T_f X$ lies in the topological center  $\Lambda_\Phi$ of $\Phi$.

Now assume that the operation $\varphi$ is associative. By Lemma~\ref{asl}, $T_f X$ is a subsemigroup of $(X,\Phi)$. Since $T_f X$ is dense and lies in the topological center $\Lambda_\Phi$, we may derive the associativity of $\Phi$ from Proposition~\ref{asp}.
\end{proof}  

\begin{problem} Given a discrete semigroup $X$ investigate the algebraic and topological properties of the compact right-topological semigroup $T\beta X$ for some concrete continuous monadic functors $T:\Comp\to\Comp$.
\end{problem}

This problem was addressed in \cite{G1}, \cite{G2} for the monadic functor $G$ of inclusion hyperspaces, in \cite{BG1}--\cite{BG4} for the functor of superextension $\lambda$, in \cite{BCHR}, \cite{Heyer}, \cite{Par} for the functor $P$ of probability measures and in \cite{BHr}, \cite{Ber}, \cite{BL}, \cite{Trn} for the hyperspace functor $\exp$. 

In \cite{Zar} it was shown that for each continuous monadic functor $T:\Comp\to\Comp$ any continuous (associative) operation $\varphi:X\times Y\to Z$ in $\Comp$ extends to a continuous (associative) operation $\Phi:TX\times TY\to TZ$.

 \begin{problem} For which monads $\IT=(T,\eta,\mu)$ in the category $\Comp$ each right-topological (associative) binary operation $\varphi:X\times Y\to Z$ in $\Comp$ extends to a right-topological (associative) binary operation $\Phi:TX\times TY\to TZ$? Are all such monads power monads?
\end{problem}

\end{document}